\documentclass[12pt, leqno]{amsart}
\usepackage{amsfonts, amssymb, amsmath, amscd, lastpage, fancyheadings, array}

\newtheorem{thm}{Theorem}
\newtheorem{theorem}{Theorem}[section]

\newtheorem{lem}[theorem]{Lemma}
\newtheorem{lemma}[theorem]{Lemma}

\newtheorem{coro}[theorem]{Corollary}

\newcommand{\al}{\alpha}
\newcommand{\be}{\beta}

\newcommand{\om}{\omega}
\newcommand{\del}{\delta}

\newcommand{\cN}{\mathcal{N}}
\newcommand{\cM}{\mathcal{M}}

\newcommand{\lc}{\big{\lceil}}
\newcommand{\rc}{\big{\rceil}}
\newcommand{\lf}{\big{\lfloor}}
\newcommand{\rf}{\big{\rfloor}}
\newcommand{\Z}{\mbox{$\mathbb Z$}}

\newcommand{\R}{\mbox{$\mathbb R$}}     

\newcommand{\A}{\mathfrak{A}}
\newcommand{\B}{\mathfrak{B}}

\begin{document}
\title{Extensions of Schur's irreducibility results}
\author{Shanta Laishram}
\email{shanta@isid.ac.in}
\address{Stat Math Unit, Indian Statistical Institute,
7 SJS Sansanwal Marg, New Delhi 110016, India}
\author[T. N. Shorey]{T. N. Shorey}
\email{shorey@math.iitb.ac.in}
\address{Department of Mathematics\\
Indian Institute of Technology Bombay, Powai, Mumbai 400076, India}
\begin{abstract}
We prove that the generalised Laguerre polynomials
$L_{n}^{(\alpha)}(x)$ with $0\le \al\le 50$ are irreducible except for finitely
many pairs $(n, \al)$ and that these exceptions are necessary. In fact it follows
from a more general statement.
\end{abstract}

\maketitle

\section{Introduction}

For $\al\in \R$ and $n\in \Z$ with $n\ge 1$, we define the generalised
Laguerre polynomials of degree $n$ as
\begin{align*}
L_{n}^{(\alpha)}(x) = \sum_{j=0}^{n} \frac{(n+\alpha)(n-1+\alpha)\cdots (j+1+\alpha)
(-x)^{j}}{(n-j)! j!}.
\end{align*}
There is an extensive literature on Laguerre polynomials. In particular, the
irreducibility of these class of orthogonal polynomials has been well studied. The
irreducibility of $L^{(-2n-1)}_n$ proved by Filaseta and Trifonov \cite{fitr} is equivalent
to the fact that all Bessel polynomials are irreducible. Also Laguerre polynomials provide examples
of polynomials of degree $n$ with associated Galois group $A_n$ where $A_n$ is the alternating group
on $n$ symbols and the irreducibility of
$L^{(n)}_n$ proved by Filaseta, Kidd and Trifonov \cite{fkt} has been used to settle
explicitly the \emph{Inverse Galois problem} that for every $n>1$ there exists an explicit
polynomial of degree $n$ with associated Galois group $A_n$. We prove
\begin{thm}\label{Laga40}
Let $0\le \al\le 50$. Then $L^{(\al)}_n(x)$ is irreducible except when
$n=2, \al\in \{2, 7, 14, 23, 34, 47\}$ and $n=4, \al \in \{5, 23\}$ where it
has a linear factor.
\end{thm}
For the exceptions, we have
\begin{align*}
\begin{array}{ll}
L^{(2)}_2(x)=\frac{1}{2}(x - 2)(x - 6);  \  \  &L^{(7)}_2(x)=\frac{1}{2}(x -6)(x - 12);\\
L^{(14)}_2(x)=\frac{1}{2}(x - 12)(x -20); \ \ & L^{(23)}_2(x)=\frac{1}{2}(x - 20)(x - 30);\\
L^{(34)}_2(x)=\frac{1}{2}(x -30)(x - 42); \ \ & L^{(47)}_2(x)=\frac{1}{2}(x -42)(x - 56); \\
L^{(5)}_4(x)=\frac{1}{24}(x - 6)(x^3 - 30x^2 + 252x - 504); & \\
L^{(23)}_4(x)=\frac{1}{24}(x-30)(x^3-78x^2+1872x-14040). &
\end{array}
\end{align*}
Theorem \ref{Laga40} is an extension of a result of Filaseta, Finch and Leidy \cite{fifile} where
they proved that $L_{n}^{(\alpha)}(x)$ is irreducible for all $n$ and $0\le \al\le 10$ except when
$(n, \al)\in \{(2, 2), (4, 5), (2, 7)\}$. Therefore we shall always assume that $\al>10$ in the
proof of Theorem \ref{Laga40}. We also consider the problem of finding factors of Laguerre
polynomials. We have
\begin{thm}\label{Laga5k}
Let $1\le k\le \frac{n}{2}$ and $0\le \al\le 5k$. Then $L^{(\al)}_n(x)$ has no factor of degree $k$ except when
$k=1, (n, \al) \in \{(2, 2), (4, 5)\}$.
\end{thm}

The Laguerre polynomials are a special case of generalizations of following class of polynomials first considered
by Schur. Let $n\ge1, a\ge 0$ and $a_0, a_1, \ldots, a_n$ be integers. The \emph{generalized Schur polynomials} are
defined as
\begin{align}\label{f(x)}
f(x):=f_{n, a}(x):=f_{n, a}(a_0, a_1, \cdots, a_n)=a_n\frac{x^n}{(n+a)!}+a_{n-1}\frac{x^n}{(n-1+a)!}+\ldots +
a_1\frac{x}{(1+a)!}+a_0\frac{1}{a!}.
\end{align}
It is easy to see that by taking
\begin{align*}
a=\al \ {\rm and} \ a_j=(-1)^j\binom{n}{j} \ {\rm for} \ 0\le j\le n,
\end{align*}
we obtain $(n+\al)!f_{\al}(x)=n!L^{(\al)}_n(x)$.

Schur \cite{schur} proved that $f(x)$ with $a=0$ and $|a_0|=|a_n|=1$ is irreducible.
He also proved in \cite{schur1} that $f(x)$ with $a=1$ and $|a_0|=|a_n|=1$ is irreducible unless
$n+1=2^r$ for some $r$ where it may have a linear factor or $n=8$ where it may have a quadratic
factor. Also for $a=2$ and many other values of $a$ the polynomial $f(x)$ may have a linear factor.
Clearly if $f(x)$ is reducible, then $f(x)$ has a factor of degree $k$ with $1\le k\le \frac{n}{2}$.
Shorey and Tijdeman \cite{stirred} proved that $f(x)$ with
$2\le k\le \frac{n}{2}$, $0\le a\le \frac{3}{2}k$ and $|a_0|=|a_n|=1$ has no factor of degree
$k$ except when
\begin{align}\label{st3/2}\begin{split}
(n, k, a)\in \{&(6, 2, 3), (7, 2, 2), (7, 2, 3), (7, 3, 3), (8, 2, 1), (8, 3, 2), \\
&(12, 3, 4), (13, 2, 3), (22, 2, 3), (46, 3, 4), (78, 2, 3)\}.
\end{split}\end{align}
Furthermore all the exceptions in \eqref{st3/2} are necessary.
They also showed that for $f(x)$ with $3\le k\le \frac{n}{2}, |a_0|=|a_n|=1$ and
$0\le a\le 10$ when $k=3, 4$ or $0\le a\le 30$ when $k\ge 5$ has no factor of degree $k$
except when
\begin{align}\label{st1030}\begin{split}
(n, k, a)\in \{&(7, 3, 3), (8, 3, 2), (12, 3, 4), (18, 4, 9), (18, 4, 10), (46, 3, 4), \\
&(56, 4, 10), (17, 5, 11), (19, 5, 9), (40, 5, 12)\}.
\end{split}
\end{align}
We extend the validity of their results as follows.
\begin{thm}\label{a3k}
Let $2\le k\le \frac{n}{2}$, $0\le a\le 5k$ and $|a_0|=|a_n|=1$. Then $f_{n, a}(x)$ has no factor
of degree $k$ except possibly when $(n, k, a)$ is given by \eqref{st3/2} or \eqref{st1030} or
\begin{align}\label{exc5k}\begin{split}
k=2, \ (n, a)\in \{& (4, 5), (6, 4), (8, 8), (12, 4), (17, 8), (21, 4), (22, 6), (23, 5), \\
&(23, 10), (24, 9), (36, 9), (43, 6), (44, 5), (46, 9), (58, 6), (59, 5), \\
&(72, 9), (73, 8), (77, 4),  (91, 9), (112, 9), (233, 10), (234, 9)\};\\
k=3, \ (n, a)\in \{&(14, 12), (17, 11), (53, 12)\}; \\
k=4, \ (n, a)\in \{&(16, 12), (17, 11), (38, 13), (39, 18)\}.
\end{split}\end{align}
\end{thm}

\begin{thm}\label{a<50}
Let $2\le k\le \frac{n}{2}$, $|a_0|=|a_n|=1$ and $0\le a\le 40$ if $k=2$ and $0\le a\le 50$ if $k\ge 3$.
Then $f_{n, a}(x)$ has no factor of degree $k$ except possibly when $(n, k, a)$ is given by \eqref{st3/2}
or \eqref{st1030} or \eqref{exc5k} or the cases $k=2$ with
\begin{align*}
n+a\le 100 \ {\rm or} \ a\in \{13, 14, 19, 33\}, n+a\in \{126, 225, 2401, 4375\}
\end{align*}
or
\begin{center}
\begin{tabular}{|c|c||c|c||c|c|} \hline
$a$ & $n+a$ & $a$ & $n+a$ & $a$ & $n+a$ \\ \hline
$12$ & $169, 729$ & $15, 16$ & $289$ & $17$ & $513$ \\ \hline
$18$ & $361, 513,  1216$ & $19, 20$ & $243$ & $21$ & $529$ \\ \hline
$21, 22$ & $121, 576$ & $24$ & $325, 625, 676$ & $27$ & $784$ \\ \hline
$28$ & $145$ & $29$ & $961$ & $31$ & $243$ \\ \hline
$32$ & $243, 289, 1089$ & $33$ & $136, 256, 289, 5832$ & $36$ & $1369$ \\ \hline
$38$ & $325, 625, 676$ & $39$ & $1025, 6561$ & $40$ & $288$ \\ \hline
\end{tabular}
\end{center}
\end{thm}
It is likely to obtain factorizations in most of these cases but we have not carried
out the computations. The following assertion follows  from Theorem \ref{a<50}.
\begin{thm}\label{a12}
The polynomial $f_{n, a}(x)$ with $a_0a_n=\pm 1, a_1=a_2=\ldots =a_{n-1}=1$ and $a\le 12$ is either
irreducible or a product of a linear polynomial times a polynomial of degree $n-1$.
factor.
\end{thm}

We shall use the results of \cite{stirred} stated above without reference in this paper. Thus
we always suppose that $a>3$ if $k=2$, $a>10$ if $k=3, 4$ and $a>30$ if $k\ge 5$ in Theorems
\ref{a3k} and \ref{a<50}. Further we observe  that Theorem \ref{a<50} with $k\ge 10$ follows
from Theorem \ref{a3k}. Also Theorem \ref{Laga5k} follows immediately from Theorem \ref{Laga40}
for $k\le 10$ and from Theorem \ref{a3k} for $k>10$. Thus it suffices to prove Theorems
\ref{Laga40}, \ref{a3k}, \ref{a<50} with $k<10$ and \ref{a12}.
The new ingredients in the proofs of our theorems are the following Irreducibility Lemma and
sharper lower estimates for the greatest prime factor of $\Delta(m, k)$ where
\begin{align}\label{2u0}
\Delta(m, k)=m(m+1)\cdots (m+k-1).
\end{align}

\begin{lem}\label{irupdate}
Let $a>0, 1\le k\le \frac{n}{2}$ and $u_0=\frac{a}{k}$.

\noindent
$(A)$ Assume that there is a prime $p\ge k+2$ with
\begin{align}\label{cond}
p|\prod^k_{i=1}(a+n-k+i), \ \ p\nmid a_0a_n
\end{align}
and
\begin{align}\label{a+1k}
p\nmid \prod^k_{i=1}(a+i).
\end{align}
Suppose
\begin{align}\label{2u0}
p\ge \min(2u_0, k+u_0)
\end{align}
or
\begin{align}\label{2k}
p>2k \ {\rm and} \ p^2-p\ge a.
\end{align}
Then $f_{n, a}(x)$ has no factor of degree $k$.

\noindent
$(B)$ If there is a prime $p\ge k+2$ with
\begin{align}\label{Lagcond}
p|\prod^k_{i=1}(n-k+i)(a+n-k+i)
\end{align}
and \eqref{a+1k} and satisfying \eqref{2u0} or \eqref{2k}, then $L^{(a)}_n(x)$ has no factor of degree $k$.
\end{lem}

We have stated Lemma \ref{irupdate} and some of the subsequent lemmas in a more general way than
required for the proof of our theorems. We prove Lemma \ref{irupdate} in Section \ref{Irproof}.
In Section $3$, we give a refinement of an argument of Erd\H{o}s and Sylvester.
In Sections $5-9$, we prove Theorems \ref{Laga40}, \ref{a3k}, \ref{a<50} and \ref{a12} by combining
Lemma \ref{irupdate} with the refinement in Section $4$, results on Grimm's conjecture (see Lemma \ref{grim})
and estimates from prime number theory.  Section $3$ contains preliminaries required for the proof of our theorems.
For any real $u>0$, let $\lf u\rf$ and $\lc u\rc$ be the floor function of $u$ and the ceiling
function of $u$, respectively. Thus $\lf u\rf$ is the greatest integer less than or equal
to $u$ and $\lc u\rc$ is the least integer exceeding $u$.

\section{Proof of Lemma \ref{irupdate}}\label{Irproof}

We will use the notations introduced in this section throughout the paper. We write
\begin{align*}
\Delta_j=\Delta(a+1, j)=(a+1)(a+2)\cdots (a+j).
\end{align*}
We observe that $q|\Delta_k$ for all primes $k<q\le \frac{a+k}{\lc u_0\rc}$ since
$a\le k\lc u_0\rc<q\lc u_0\rc\le a+k$. Suppose there is a prime $p$ satisfying the condition
of the lemma. Then $p>\frac{a+k}{\lc u_0\rc}$ by \eqref{a+1k}. As in the proof  of \cite[Lemma 4.2]{stirred},
it suffices to show that
\begin{align}\label{Delfa}
\phi_j:=\phi_j(p):=\frac{{\rm ord}_p(\Delta_j)}{j}<\frac{1}{k} 
\ \ {\rm for} \ 1\le j\le n
\end{align}
for showing that $f_{n, a}(x)$ has no factor of degree $k$. Also as in the proof of \cite[Lemma 2.4]{fifile},
for showing $L^{(a)}_n(x)$ has no factor of degree $k$, it suffices to show
\begin{align}\label{DelLag}
\phi'_j:=\phi'_j(p):=\frac{{\rm ord}_p\left(\frac{\Delta_j}{\binom{n}{j}}\right)}{j}<\frac{1}{k} \ \ {\rm for} \ 1\le j\le n.
\end{align}
Since $\phi'_j\le \phi_j$, we show that \eqref{Delfa} holds for all $j$.

Let $j_0$ be the minimum $j$ such that $p|(a+j)$ and write $a+j_0=pl_0$ for some $l_0$. Then
$j_0\le p$ and $j_0>k$ since $p\nmid \Delta_k$. Also we see that $l_0\le \lc u_0\rc$ which we shall
use in the proof without reference.

We may restrict to those $j$ such that $a+j=pl$ for some $l$. Then $j-j_0=p(l-l_0)$.
Writing $l=l_0+s$, we get $j=j_0+ps$. Note that if $p|(a+j)$, then $a+j=p(l_0+r)$
for some $r$. Hence we have
\begin{align}\label{l_or}
{\rm ord}_p(\Delta _j)&={\rm ord}_p((pl_0)(p(l_0+1))\cdots (p(l_0+s)))=
s+1+{\rm ord}_p(l_0(l_0+1)\cdots (l_0+s)).
\end{align}
Let $r_0$ be such that ord$_p(l_0+r_0)$ is maximal. We consider two cases.

\noindent
{\bf Case I:} Assume that $l_0+s<p^2$. If $p\nmid (l_0+i)$ for $0\le i\le s$, then
$\phi_j=\frac{s+1}{j_0+ps}<\frac{s+1}{k+ks}=\frac{1}{k}$. Hence we may suppose that
$p|(l_0+i)$ for some $0\le i\le s$ and further $l_0+s=pl_1$ for some $1\le l_1<p$.
Assume $s=0$. Then $p|l_0$ which together with $l_0<p^2$ implies
ord$_p(\Delta _j)=$ord$_p(a+j_0)=2$. Therefore $a+p\ge a+j_0\ge p^2$ implying
$a\ge p^2-p$. If \eqref{2u0} holds, then $a\le \max(k(p-k), \frac{pk}{2})<p(p-1)$
which is not possible. Thus \eqref{2k} holds and hence $p\ge 2k+1$ and $a=p^2-p$
implying $j_0=p$. Therefore $\phi_j=\frac{2}{j_0}=\frac{2}{p}<\frac{1}{k}$.
Thus we have $s\neq 0$ and we obtain from \eqref{l_or} that
ord$_p(\Delta _j)=s+1+l_1$ implying $\phi_j\le \frac{s+1+l_1}{j_0+ps}$.
Hence $\phi_j<\frac{1}{k}$ if $(p-k)\frac{s}{l_1}\ge k$ since
$\frac{j_0+sp}{k}>1+s\frac{p}{k}$.

Suppose $p$ satisfies \eqref{2k}. Then we may assume that $s<l_1$. Since $l_1<p$, we have $s<p$ implying
ord$_p(\Delta _j)\le s+2$ giving $\phi_j<\frac{s+2}{k+ps}\le \frac{1}{k}$ since $s>0$.

Thus we assume that $p$ satisfies \eqref{2u0}. Since $p\ge k+2, s=pl_1-l_0$ and $l_0\le \lc u_0\rc$,
we have $(p-k)\frac{s}{l_1}-k\ge 2(p-\frac{l_0}{l_1})-k\ge 2p-k-2\lc u_0\rc$. Hence it suffices to
show $2p-k\ge 2\lc u_0\rc$. Since $p\ge \min(2u_0, k+u_0)$, we have
\begin{align*}
2p-k=p+p-k\ge \begin{cases}
2u_0+2\ge 2\lc u_0\rc & {\rm if} \ p\ge 2u_0\\
2(k+\lc u_0\rc)-k\ge 2\lc u_0\rc & {\rm if} \ p\ge k+u_0,
\end{cases}
\end{align*}
noting that $p\ge k+u_0$ implies $p\ge k+\lc u_0\rc$.

\noindent
{\bf Case II:} Let $l_0+s\ge p^2$. Then we get from \eqref{l_or} that
\begin{align*}
{\rm ord}_p(\Delta _j)\le s+1+{\rm ord}_p(l_0+r_0)+ {\rm ord}_p(s!)\le s+1+\frac{\log (l_0+s)}{\log p}+\frac{s}{p-1}.
\end{align*}
Since $\frac{j}{k}=\frac{j_0+ps}{k}>1+\frac{p}{k}s$, it is enough to show that
\begin{align*}
\frac{p}{k}\ge 1+\frac{1}{p-1}+\frac{\log (l_0+s)}{s\log p}.
\end{align*}
Observe that $\frac{\log (l_0+s)}{s\log p}$ is a decreasing function of $s$. Since $s\ge p^2-l_0$,
it suffices to show
\begin{align*}
\frac{p}{k}\ge 1+\frac{1}{p-1}+\frac{2}{p^2-l_0}.
\end{align*}
Suppose $p$ satisfies \eqref{2u0}. Then from $l_0\le \lc u_0\rc \le p$ and $p\ge k+2$, we have $p^2-l_0\ge (k+2)^2-(k+2)\ge 2(k+1)$
implying
\begin{align*}
1+\frac{1}{p-1}+\frac{2}{p^2-l_0}\le 1+\frac{1}{k+1}+\frac{2}{2(k+1)}<1+\frac{2}{k}\le \frac{p}{k}.
\end{align*}
Suppose $p$ satisfies \eqref{2k}. Then from $l_0\le \lc u_0\rc \le a$ and $p>2k$, we obtain
$p^2-l_0\ge p^2-a\ge p>2k$ implying
\begin{align*}
1+\frac{1}{p-1}+\frac{2}{p^2-l_0}\le 1+\frac{1}{2k}+\frac{2}{2k}<1+\frac{2}{k}\le \frac{p}{k}.
\end{align*}
Hence the assertion.
$\hfill \Box$

\begin{coro}\label{ircor3}
Let $k, p$ and $\A_{k, p}$ be given by
\begin{align*}
&k=1, \ p=3, \ \A_{1, 3}=\{3r, 3r+1: 0\le r\le 16\}\setminus \{7, 16, 24, 25, 34, 43\}\\
&k=1, \ p=5, \ \A_{1, 5}=\{5r, 5r+1, 5r+2, 5r+3: 0\le r\le 9\}\cup \{50\}\setminus \{23, 48\}\\
&k=1, \ p=7, \ \A_{1, 7}=[0, 50]\cap \Z\setminus \{6, 13, 20, 27, 34, 41, 47, 48\}\\
&k=2, \ p=5, \ \A_{2, 5}= \{5r, 5r+1, 5r+2: 0\le r\le 8\}\cup \{45, 50\}\setminus \{21, 22\}\\
&k=2, \ p=7, \ \A_{2, 7}=[0, 50]\cap \Z\setminus \left(\{7r-1, 7r-2: 1\le r\le 7\}\cup \{45, 46\}\right)\\
&k=3, \ p=5, \ \A_{3, 5}=\{0, 1, 5, 6, 10, 11, 15, 25, 26, 30, 31, 35, 36, 40, 50\}\\
&k=3, \ p=7, \ \A_{3, 7}=\{7r, 7r+1, 7r+2, 7r+3: 0\le r\le 5\}\cup \{42, 49, 50\}\\
&k=4, \ p=7, \ \A_{4, 7}=\{7r, 7r+1, 7r+2: 0\le r\le 4\}\cup \{35, 36, 49, 50\}\\
&k=5, \ p=7, \ \A_{5, 7}=\{0, 1, 7, 8, 14, 15, 21, 22, 28, 49, 50\}.
\end{align*}
Suppose $n\ge 2k$ and $p$ satisfies \eqref{cond}. Then $f_{n, a}(x)$ has no factor of degree $k$ for $a\in \A_{k, p}$.
Further if $p$ satisfies \eqref{Lagcond}, then $L^{(a)}_n(x)$ has no factor of degree $k$ for $a\in \A_{k, p}$.
\end{coro}

\begin{proof}
For $k, p$ and $a\in \A_{k, p}$ given in the statement of Corollary \ref{ircor3}, we check that
$p\nmid \Delta_k$ and $\frac{{\rm ord}_p(\Delta_j)}{j}<\frac{1}{k}$ for $j\le 50$.
As in the proof of Lemma \ref{irupdate}, it suffices to check that
$\frac{{\rm ord}_p(\Delta_j)}{j}<\frac{1}{k}$ for all $j\ge 1$. Since ord$_p(s!)\le \frac{s}{p-1}$,
we have for $j>50$ that
\begin{align*}
\frac{{\rm ord}_p(\Delta_j)}{j}=\frac{{\rm ord}_p((a+j)!)-{\rm ord}_p(a!)}{j}
\le \frac{\frac{a+j}{p-1}-{\rm ord}_p(a!)}{j}\le \frac{1}{p-1}+\frac{\frac{a}{p-1}-{\rm ord}_p(a!)}{51}<\frac{1}{k}.
\end{align*}
Thus $\frac{{\rm ord}_p(\Delta_j)}{j}<\frac{1}{k}$ for all $j\ge 1$.
\end{proof}

\begin{coro}\label{Ircor}
Let $a>0$ and $1\le k\le \frac{n}{2}$.
\begin{itemize}
\item[$(i)$] If there is a prime $p>a+k$ satisfying \eqref{cond},
then $f_{n, a}(x)$ has no factor of degree $k$.
\item[$(ii)$] Let $p\ge k+2$ be a prime satisfying \eqref{cond} and let
\begin{align*}
{\mathcal A}_p:=\bigcup^{r_p}_{i=1}\left([ip-k, ip-1]\cap \Z_{>0}\right)\cup \{j>pr_p, j\in \Z\}
\end{align*}
where
\begin{align*}
r_p=\lf \frac{k}{2}\rf \ \ {\rm if} \ \ p<2k \ {\rm and} \ p-1 \ \ {\rm if} \ \ p\ge 2k.
\end{align*}
Then $f_{n, a}(x)$ has no factor of degree $k$ for $a\notin {\mathcal A}_p$.
\item[$(iii)$] Let $P_1>P_2>\ldots >P_s\ge k+2$ be primes satisfying \eqref{cond}. For a subset
$\{Q_1, Q_2, \ldots, Q_g\}\subseteq \{P_1, P_2, \ldots , P_s\}$, let
\begin{align*}
{\mathcal B}\{Q_1, \ldots, Q_g\}=\bigcap^{g}_{l=1}\mathcal{A}_{Q_l}.
\end{align*}
Then $f_{n, a}(x)$ has no factor of degree $k$ for $a\notin \mathcal{B}\{Q_1, \ldots, Q_g\}$.
\end{itemize}
\end{coro}

In earlier results, Corollary \ref{Ircor} $(i)$ has been used. This is possible only if there is a
$p>k+a$ satisfying \eqref{cond}. But it is possible to apply Lemma \ref{irupdate} even when $p\le k+a$
for all $p$ satisfying \eqref{cond}. For example, take $n=15, a=13, k=3$. Here $p<k+a$ for all $p$
satisfying \eqref{cond}. However \eqref{cond}, \eqref{a+1k} and \eqref{2k} are satisfied with $p=13$ and
hence $f_{n, 13}(x)$ has no factor of degree $3$ by Lemma \ref{irupdate}.
\begin{proof}
$(i)$ is immediate from Lemma \ref{irupdate}.
Consider $(ii)$. We may assume that $p\le k+a$ by $(i)$. Let $a\notin \mathcal{A}_p$. Then
$a\le pr_p$ implying $a\le p^2-p$ if $p\ge 2k$ and $2u_0=\frac{2a}{k}\le \frac{2pr_p}{k}\le p$
if $p<2k$ satisfying either \eqref{2u0} or \eqref{2k}. Since $a\notin \mathcal{A}_p$, there is
some $i$ for which $ip-1<a<(i+1)p-k$ implying $ip<a+1<a+k<(i+1)p$. Therefore
$p\nmid \prod^k_{j=1}(a+j)$ which together with \eqref{cond} and $p\ge k+2$ satisfy the
conditions of Lemma \ref{irupdate}. Now the assertion follows by Lemma \ref{irupdate}.
The assertion $(iii)$ follows from $(ii)$.
\end{proof}

\section{Preliminaries for Theorems \ref{a3k}-\ref{a12}}

For a positive integer $\nu >1$, we denote by $\om(\nu)$ and $P(\nu)$ the number of distinct prime
factors and the greatest prime factor of $\nu$, respectively, and we put $\om(1)=0, P(1)=1$.
For positive integers $\nu$, we write
\begin{align*}
\pi(\nu)=&\sum_{p\le \nu} 1, \\
\theta(\nu)=&\sum_{p\le \nu} \log p.
\end{align*}
Let $p_i$ denote the $i-th$ prime.

We begin with some results on primes.

\begin{lemma}\label{pix} Let $k\in \Z$ and $\nu\in \R$. We have
\begin{itemize}
\item[(i)] $\pi(\nu)\geq \frac{\nu}{\log \nu -1}$ for $\nu \ge 5393$ and $\pi(\nu)\le
\frac{\nu}{\log \nu} \left(1+\frac{1.2762}{\log \nu}\right)$ for $\nu>1$.
\item[(ii)] $\pi(\nu_1+\nu_2) \le \pi(\nu_1)+\pi(\nu_2)$ for
$2\le \nu_1<\nu_2\le \frac{7}{5}\nu_1(\log \nu_1)(\log \log \nu_1)$.
\item[(iii)] $\nu(1-\frac{3.965}{\log^2 \nu})\le \theta (\nu)<1.00008\nu$ for $\nu>1$.
\item[(iv)] $p_k\ge k\log k \ {\rm for} \ k\ge 1$.
\item[(v)] {\rm ord}$_p((k-1)!) \geq \frac{k-p}{p-1}-\frac{\log (k-1)}{\log p}$ for $k\ge 2$.
\item [(vi)] $\sqrt{2\pi k}~e^{-k}k^{k}e^{\frac{1}{12k+1}} <k!<
\sqrt{2\pi k}~e^{-k}k^{k} e^{\frac{1}{12k}}$.
\end{itemize}
\end{lemma}

The estimates $(i), (ii)$ and $(iii)$ are due to Dusart (\cite{dus1} and \cite{dus2}, respectively). The
estimate $(iv)$ is due to Rosser \cite{Ros} and estimate $(vi)$ is due to Robbins \cite[Theorem 6]{robb}.
For a proof of $(v)$, see \cite[Lemma 2(i)]{shanta2}.
\qed

We derive from Lemma \ref{pix} the following results.

\begin{coro}\label{n<<}
Let $10^{10}<m\le 123k$. Then there are primes $p, q$ with
$m\le p<m+k$ and $\frac{m}{2}\le q<\frac{m+k}{2}$.
\end{coro}

\begin{proof}
Let $10^{10}<m\le 123k$. We observe that the assertion holds if
\begin{align*}
\theta(\frac{m+k-1}{s})-\theta(\frac{m-1}{s})=\sum_{\frac{m-1}{s}<p\le \frac{m+k-1}{s}}\log p >0
\end{align*}
for $s=1, 2$. Now from Lemma \ref{pix} and since $m>10^{10}$, it suffices to show
\begin{align*}
\theta(\frac{m+k-1}{s})-\theta(\frac{m-1}{s})>\frac{m+k-1}{s}
\left(1-\frac{3.965}{\log^2 (5\cdot 10^9)}\right)-1.00008\frac{m-1}{s}>0
\end{align*}
or
\begin{align*}
k(1-\frac{3.965}{\log^2 (5\cdot 10^9)})>(m-1)(\frac{8}{10^5}+\frac{3.965}{\log^2 (5\cdot 10^9)}).
\end{align*}
This is true since $m\le 123k$ and
\begin{align*}
\frac{1-\frac{3.965}{\log^2 (5\cdot 10^9)}}{\frac{8}{10^5}+\frac{3.965}{\log^2 (5\cdot 10^9)}}>123.
\end{align*}
\end{proof}

\begin{coro}\label{<4k8000}
We have
\begin{align}\label{<4k80001}
\pi(k)+\pi(\frac{k}{2})+\pi(\frac{k}{3})+\pi(\frac{k}{4})+\pi(\frac{6k}{5})\le \begin{cases}
k-2 & {\rm for} \ k\ge 61\\
\pi(4k) \ & {\rm for} \ k\ge 8000.
\end{cases}
\end{align}
\end{coro}

\begin{proof}
Let $k\ge 30000$. We have from $\frac{\log y}{\log x}=1+\frac{\log y/x}{\log x}$ and
Lemma \ref{pix} $(i)$ that
\begin{align*}
&(\log 4k)\left(\pi(4k)-\pi(\frac{6k}{5})-\pi(k)-\pi(\frac{k}{2})-\pi(\frac{k}{3})-
\pi(\frac{k}{4})\right)\\
\ge &\frac{4k}{\log 4k-1}+\\
&k\left(4-\frac{6}{5}\left(1+\frac{\log \frac{10}{3}}{\log \frac{6k}{5}}\right)
\left(1+\frac{1.2762}{\log \frac{6k}{5}}\right)-
\sum^{4}_{j=1}\frac{1}{j}\left(1+\frac{\log 4j}{\log \frac{k}{j}}\right)
\left(1+\frac{1.2762}{\log \frac{k}{j}}\right)\right).
\end{align*}
The right hand side of the above inequality is an increasing function of $k$ and it is positive
at $k=30000$. Therefore the left hand side of \eqref{<4k80001} is at most $\pi(4k)$ for $k\ge 30000$. By
using exact values, we find that it is valid for $k\ge 8000$.

Also $\pi(4k)\le \frac{4k}{\log 4k}\left(1+\frac{1.2762}{\log 4k}\right)\le k-2$ is true for $k\ge 8000$.
Therefore the left hand side of \eqref{<4k80001} is at most $k-2$ for $k\ge 8000$. Finally we check using
exact values of the $\pi-$function that the left hand side of \eqref{<4k80001} is at most $k-2$ for
$61\le k<8000$.
\end{proof}

The following result is on Grimm's Conjecture, \cite[Theorem 1]{grim}. Grimm's Conjecture
states that \emph{given integers $n\ge 1$ and $k\ge 1$ such that whenever $n+1, \cdots , n+k$ are
all composite numbers, we can find distinct primes $P_i$ with $P_i|(n+i)$ for $1\le i\le k$.}
This is a difficult conjecture having several interesting consequences.
For example, this conjecture implies $p_{i+1}-p_i<p^{0.46}_i$ for sufficiently large $i$, a result
better than that given by Riemann hypothesis. This follows by taking $n=p_i$ in \cite[Theorem 1(i)]{lmgrim}.
We refer to \cite{rstgrim} and \cite{lmgrim} for a survey and results on Grimm's Conjecture.

\begin{lemma}\label{grim}
Let $m\le 1.9\cdot 10^{10}$ and $l\ge 1$ be such that $m+1, m+2, \cdots , m+l$ are all
composite numbers. Then there are distinct primes $P_i$ such that
$P_i|(m+i)$ for each $1\le i\le l$.
\end{lemma}

The following result follows from \cite[Lemma 3]{Sar}.

\begin{lemma}\label{k3/2}
Let $m+k-1<k^{\frac{3}{2}}$. Let $|\{i: P(m+i)\le k\}|=\mu$. Then
\begin{align*}
\binom{m+k-1}{k}\leq (2.83)^{k+\sqrt{m+k-1}}(m+k-1)^{k-\mu}.
\end{align*}
\end{lemma}

\section{An upper bound for $m$ when $\om(\Delta (m, k))\le t$}\label{omdbd}

Let $m, k$ and $t$ be positive integers such that
\begin{align}\label{piDk}
\om(\Delta (m, k))\le t.
\end{align}
For every prime $p$ dividing $\Delta(m, k)$, we delete a term $m+i_p$ in $\Delta(m, k)$
such that ord$_p(m+i_p)$ is maximal. Then we have a set $T$ of terms in
$\Delta(m, k)$ with
\begin{align*}
|T|=k-t:=t_0.
\end{align*}
We arrange the elements of $T$ as $m+i_1<m+i_2<\cdots <m+i_{t_0}$. Let
\begin{align}\label{P0}
{\frak P}:=\displaystyle{\prod^{t_0}_{\nu =1}} (m+i_{\nu}) 
\ge m^{t_0}.
\end{align}

Now we obtain an upper bound for ${\frak P}$. For a prime $p$, let $r$ be
the highest power of $p$ such that $p^r\le k-1$ and let $i_0$ be such that
ord$_p(m+i_0d)$ is maximal. Let $w_l=|\{m+i: p^l|(m+i), m+i\in T\}|$
for $1\le l\le r$. By an argument that was first given by Sylvester and Erd\H{o}s(see \cite{}),
we have $w_l\le [\frac{i_0}{p^l}]+[\frac{k-1-i_o}{p^l}]\le [\frac{k-1}{p^l}]$.
Let $h_p>0$ be such that
$[\frac{k-1}{p^{h_p+1}}]\le t_0<[\frac{k-1}{p^{h_p}}]$. Then there are at most
$t_0-w_{h_p+1}$ terms in $T$ exactly divisible by $p^l$ with $l\le h_p$. Hence
\begin{align*}
{\rm ord}_p({\frak P})&\le rw_r +
\sum^{r-1}_{u=h_p+1}u(w_u-w_{u+1})+h_p(t_0-w_{h_p+1})\\
&=w_r+w_{r-1}+\cdots +w_{h_p+1}+h_pt_0\\
&\le \sum^{r}_{u=1}\lf\frac{k-1}{p^u}\rf +h_pt_0-
\sum^{h_p}_{u=1}\lf \frac{k-1}{p^u}\rf ={\rm ord}_p((k-1)!)+
h_pt_0-\sum^{h_p}_{u=1}\lf \frac{k-1}{p^u}\rf .
\end{align*}
It is also easy to see that ord$_p({\frak P})\le $ord$_p((k-1)!)$. Let
$L_0(p)=$min$(0, h_pt_0-\sum^{h_p}_{u=1}\lf \frac{k-1}{p^u}\rf )$. For any $l\ge 1$,
we have from \eqref{P0} that
\begin{align}\label{Lbd}
m\le \left({\mathfrak P}\right)^{\frac{1}{t_0}} \le
\left( (k-1)! \prod_{p\le p_l}p^{L_0(p)}\right)^{\frac{1}{t_0}}=:L(k, l).
\end{align}
Observe that
\begin{align}\label{E0}
m^{t_0}\le (L(k, l))^{t_0}\le (k-1)!.
\end{align}

\section{Prelude to the proof of Theorems \ref{a3k}-\ref{a12}}

Let $k\ge 2$, $n\ge 2k$, $a\ge 0, m=n+a-k+1$ and $|a_0a_n|=1$. Then $m>k+a$.
We consider the polynomials $f_{n, a}(x)$ with
$3<a\le 40$ when $k=2$; $10<a\le 50$ when $k\in \{3, 4\}$ and $\max(30, 1.5k)<a\le \max(50, 5k)$ when $k\ge 5$.
Let $P_1>P_2>\ldots >P_s\ge k+2$ be primes dividing $\Delta(m, k)$. We write $P_{m, k}=\{P_1, P_2, \ldots ,P_s\}$.
We use Corollaries \ref{ircor3} and \ref{Ircor} to apply the following procedure which we refer
to as \emph{Procedure $\mathcal{R}$}.

\noindent
{\bf Procedure $\mathcal{R}$:} Let $k$ be fixed. For all $a$ with $3<a\le 40$ if $k=2$;
$10<a\le 50$ if $k\in \{3, 4\}$ and $\max(30, 1.5k)<a\le \max(50, 5k)$ if $k\ge 5$, it suffices
to consider only $(m, k, a)$ with $P_1\le k+a$ by Corollary \ref{Ircor} $(i)$.
We restrict to such triples $(m, k, a)$ with $P_1\le k+a$. By Corollary \ref{Ircor} $(iii)$,  we have
$a\in \B_0(m, k):=\B\{P_1, P_2, \ldots ,P_s\}$. Therefore we further restrict to $(m, k, a)$ with
$a\in \B_0(m, k)$. Further for $k\in \{2, 3, 4, 5\}$ and $p=5\in P_{m, k}$ if $k=2$;
$p=5\in P_{m, k}$ or $p=7\in P_{m, k}$ if $k=3$ and $p=7\in P_{m, k}$ if $k\in \{4, 5\}$,
we restrict to those $(m, k, a)$ with $a\notin \A_{k, p}$
by using Corollary \ref{ircor3} and recalling $n=m+k-1-a$. Every $(m, k, a)$ gives rise
to the triplet $(n, k, a)$.

We try to exclude the triplets $(n, k, a)$ given by \emph{Procedure $\mathcal{R}$} to prove our theorems.

Let
\begin{align*}
\om_0(a)=\begin{cases}
\pi(a+k) & {\rm if} \ a\le k+1\\
\sum^2_{j=1}\left(\pi(\frac{a+k}{j})-\pi(\max(k+1, \frac{a}{j}))\right)+\pi(k+1) & {\rm if} \ k+1<a\le 2k+2\\
\sum^3_{j=1}\left(\pi(\frac{a+k}{j})-\pi(\max(k+1, \frac{a}{j}))\right)+\pi(k+1) & {\rm if} \ 2k+2<a\le 3k+3\\
\sum^4_{j=1}\left(\pi(\frac{a+k}{j})-\pi(\max(k+1, \frac{a}{j}))\right)+\pi(k+1) & {\rm if} \ 3k+3<a\le 4k+4\\
\sum^5_{j=1}\left(\pi(\frac{a+k}{j})-\pi(\max(k+1, \frac{a}{j}))\right)+\pi(k+1) & {\rm if} \ 4k+4<a\le 5k
\end{cases}
\end{align*}
and $\om_1$ be the maximum of $\om_0(a)$ for $1.5k<a\le 5k$. Then $\om(\Delta(a+1, k))\le \om_1$.

Let $k\ge 10$. Assume that $\om(\Delta(m, k))>\om_1$. Then there is a prime
$p\ge k+2$ with $p|\Delta(m, k)$ such that $p\nmid \Delta(a+1, k)$ and $p\nmid a_0a_n$. Further
$p\ge 13>2u_0$ since $u_0\le 5$. Hence $f(x)$ has no factor of degree $k$ by Lemma \ref{irupdate}.
Therefore we may suppose that
\begin{align}\label{1+om1}
\om(\Delta(m, k))\le \om_1 \ {\rm for} \ k\ge 10.
\end{align}

Let $k\ge 100$. Let $(i-1)(k+1)<a\le i(k+1)$ with $1\le i\le 5$. For $1\le j<i$, we have
$\frac{a}{j}>\frac{k}{j}\ge \frac{100}{4}$ implying
$\frac{\frac{a}{j}}{\frac{k}{j}}=\frac{a}{k}\le 5\le \frac{7}{5}\log(25)\log\log(25)
\le \frac{7}{5}\log(\frac{k}{j})\log\log(\frac{k}{j})$. Hence
$\pi(\frac{a+k}{j})-\pi(\frac{a}{j})\le \pi(\frac{k}{j})$ for $1\le j<i$ by Lemma \ref{pix} $(ii)$.
Therefore
\begin{align*}
\om_0(a)\le \begin{cases}
\pi(k+k+1) & {\rm if} \ a\le k+1\\
\pi(k)+\pi(\frac{k}{2}+k+1) & {\rm if} \ k+1<a\le 2k+2\\
\pi(k)+\pi(\frac{k}{2})+\pi(\frac{k}{3}+k+1) & {\rm if} \ 2k+2<a\le 3k+3\\
\pi(k)+\pi(\frac{k}{2})+\pi(\frac{k}{3})+\pi(\frac{k}{4}+k+1) & {\rm if} \ 3k+3<a\le 4k+4\\
\pi(k)+\pi(\frac{k}{2})+\pi(\frac{k}{3})+\pi(\frac{k}{4})+\pi(\frac{k}{5}+k) & {\rm if} \ 4k+4<a\le 5k
\end{cases}
\end{align*}
which, again by Lemma \ref{pix} $(ii)$, implies
\begin{align}\label{om2}
\om_1\le \pi(k)+\pi(\frac{k}{2})+\pi(\frac{k}{3})+\pi(\frac{k}{4})+\pi(\frac{6k}{5})
=:\om_2 \ {\rm for} \ k\ge 100.
\end{align}

Let $N_1(p)=\{N: P(N(N-1))\le p\}$ and $N_2(p)=\{N: P(N(N-2))\le p, N \ {\rm odd}\}$. Then
$N_1$ and $N_2$ are given by \cite[Table IA]{lehmer} for $p\le 41$ and \cite[Table IIA]{lehmer}
for $p\le 31$, respectively and we shall use them without reference. For given $k, N$ and $j$ with
$1\le j<k$, we put
\begin{align*}
M_j(N, k)=\prod^{k-1}_{i=0}(N-j+i).
\end{align*}
Let
\begin{align*}
\cN_j(k):=\{N\in N_1(41): P(M_j(N, k))\le 59\}.
\end{align*}
By observing that
\begin{align*}
M_1(N, k+1)=M_1(N, k)(N-1+k), \ M_k(N, k+1)=(N-k)M_{k-1}(N, k)
\end{align*}
and
\begin{align*}
M_j(N, k+1)=M_j(N, k)(N-j+k)=(N-j)M_{j-1}(N, k) \  \ {\rm for} \ 1<j<k,
\end{align*}
we can compute $\cN_j(k)$ recursively as follows. Recall that $P(N(N-1))\le 41$ for $N\in N_1(41)$.
Hence we have
\begin{align*}
\cN_1(3)=\{N\in N_1(41): P(N+1)\le 59\}, \ \cN_2(3)=\{N\in N_1(41): P(N-2)\le 59\}.
\end{align*}
For $k\ge 3$ and $1\le j\le k$, we obtain $\cN_j(k+1)$ recursively by
\begin{align*}
\cN_1(k+1)=\{N\in \cN_1(k): P(N-1+k)\le 59\}, \
\cN_k(k+1)=\{N\in \cN_{k-1}(k): P(N-k)\le 59\}
\end{align*}
and
\begin{align*}
\cN_j(k+1)=\{N\in \cN_j(k): P(N-j+k)\le 59\}\cup \{N\in \cN_{j-1}(k): P(N-j)\le 59\} \ {\rm for} \ 1<j<k.
\end{align*}

\section{Proof of Theorems \ref{a3k} and \ref{a<50} for $k<10$}\label{a<50Pf}

Let $k=2$. Then $a\le 40$. By Corollary \ref{Ircor} $(i)$, we first restrict to those $m$
for which $P(m(m+1))\le 41$. They are given by $m=N-1$ with $N\in N_1(41)$. By \emph{Procedure $\mathcal{R}$},
we obtain the tuples $(n, 2, a)$ given in the following table.
\begin{center}
\begin{tabular}{|c|c||c|c||c|c|} \hline
$a$ & $n+a$ & $a$ & $n+a$ & $a$ & $n+a$ \\ \hline
$4, 5$ & $9$ & $4$ & $10$ & $5, 6$ & $28, 49, 64$ \\ \hline
$4, 8, 9$ & $16, 25, 81$ & $9$ & $33, 45, 55, 100, 121, 243$ & $10$ & $33, 243$ \\ \hline
$12$ & $27, 28, 49, 64, 91, 169, 729$ & $13$ & $21, 25, 28, 36, 50, 64$ & $14$ & $25$\\ \hline
$13, 14$ & $81, 126, 225, 2401, 4375$ & $15, 16$ & $289$ & $17$ & $513$ \\
$19, 33$ & & & & & \\ \hline
$18$ & $25, 76, 81, 96, 361, 513, 1216$ & $19$ & $25, 28, 36,  49, 50, 64, 243$ & $20$ & $28, 33, 49, 64, 243$ \\ \hline
$21$ & $25, 33, 45, 55, 529$ & $21, 22$ & $46, 81, 100, 121, 576$ & $23$ & $81$ \\ \hline
$24$ & $40, 81, 65, 325, 625, 676$ & $26$ & $49, 64$ & $27$ & $49, 64, 784$ \\ \hline
$28$ & $81, 145$ & $29$ & $81, 125, 961$ & $31$ & $243$ \\ \hline
$32$ & $243, 289, 1089$ & $33$ & $49, 50, 51, 64, 85,$ & $34$ & $49, 50, 64, 81$ \\
& & & $136, 256, 289, 5832$ & & \\ \hline
$36$ & $1369$ & $38$ & $65, 81, 325, 625, 676$ & $39$ & $81, 82, 1025, 6561$ \\ \hline
$40$ & $49, 64, 82, 288$ & & & & \\ \hline
\end{tabular}
\end{center}

Let $3\le k\le 9$. Then $10<a\le 50$ if $k=3, 4$ and $30<a\le 50$ if $5\le k\le 9$. Thus we may
assume that $P(\Delta(m, k))\le 59$ by Corollary \ref{Ircor} $(i)$.

Let $m\le 10000$. We need to consider $[k, 59]\cup \cM(k)$ where
$\cM(k)=\{60\le m\le 10000: P(\Delta(m, k))\le 59\}$.
We compute $\cM(3)$ and further from the identity $\Delta(m, k+1)=(m+k)\Delta(m, k)$, we obtain
$\cM(k+1)=\{m\in \cM(k): P(m+k)\le 59\}$ for $k\ge 3$ recursively. In fact we get
\begin{align*}
\cM(6)=\{90, 91, 116, 184, 185, 285, 340\}, \ \ \cM(7)=\{90, 184\}
\end{align*}
and $\cM(8)=\cM(9)=\emptyset$.
We now apply Procedure $\mathcal{R}$ on $m\in [k, 59]\cup \cM(k)$. We get
\begin{center}
\begin{tabular}{|c|c||c|c|} \hline
$a$ & $n+a$ & $a$ & $n+a$ \\ \hline
$11$ & $28$ & $12$ & $26, 27, 28, 65$ \\ \hline
$19, 20$ & $56, 100$ & $20$ & $46, 162$ \\ \hline
$21$ & $46$ & $32$ & $51, 56, 100, 121$ \\ \hline
$33$ & $51$ & $38, 39$ & $82$ \\ \hline
$41, 43$ & $56, 100$ & $43, 44, 45$ & $162$ \\ \hline
\end{tabular}
\end{center}
or $a\in \{12, 13, 18, 19, 20, 27, 32, 33, 34, 39, 41, 43, 44\}$, $n+a=50$ if $k=3$ and
\begin{center}
\begin{tabular}{|c|c||c|c||c|c||c|c|} \hline
$a$ & $n+a$ & $a$ & $n+a$ & $a$ & $n+a$ & $a$ & $n+a$ \\ \hline
$11, 12$ & $27, 28$ & $13, 31, 32, 33$ & $51$ & $18$ & $57$  & $10$ & $66$ \\ \hline
\end{tabular}
\end{center}
if $k=4$.

Thus $m>10000$. Suppose that $m+j=N\in N_1(41)$ for some $1\le j<k$. Then $\Delta(m, k)=M_j(N, k)$
which implies $N\in \cN_j(k)$ since $P(\Delta(m, k))\le 59$. Let
$\cN'_j(k)=\{m\in \cN_j(k): m>10000\}$. We find that
\begin{align*}
&\cN'_1(3)=\{13311, 13455, 17576, 17577, 19551, 29601, 32799, 212381\}\\
&\cN'_2(3)=\{10881, 11662, 13312, 13456, 13690, 16170, 17577, 23375, 27456, 31213,
134850, 212382, 1205646\}\\
&\cN'_1(4)=\{17576\}, \ \cN'_2(4)=\{17577\}, \ \cN'_3(4)=\{10881\}
\end{align*}
and $\cN'_j(k)=\emptyset$ for $k\ge 5$ and $1\le j<k$. We now take $m=N-j$ with $N\in \cN_j(k)$
for $1\le j<k$ and apply Procedure $\mathcal{R}$ to find that there are no triplets $(n, k, a)$.

Thus we may suppose that $m+j\notin N_1(41)$ for all $1\le j<k$. Then $P((m+i)(m+i+1))>41$ for each
$0\le i<k-1$. By Corollary \ref{Ircor} $(i)$, we may suppose that $P(\Delta (m, k))\le 53$ for
$k\le 8$ and $P(\Delta (m, k))\le 59$ for $k=9$. Taking
$V(m, k)=\{P((m+2i)(m+2i+1)): 0\le i<\frac{k}{2}\}$, we have
$V(m, k)\subseteq \{43, 47, 53\}$ for $4\le k\le 7$ and $V(m, k)=\{43, 47, 53, 59\}$ if $k=8, 9$. Then
$k\neq 8$ and computing $\{a\le 50: a\in \B\{Q_1, Q_2\}$ for
$(Q_1, Q_2)\in \{(47, 43), (53, 43, (53, 53)\}\}$ if $k=4, 5$; $(Q_1, Q_2)=(53, 43)$ if $k=6, 7, 9$,
we find that the set is empty except when $k=5, (Q_1, Q_2)=(43, 47)$ where it is $\{42\}$.
Thus we may assume that $k=5$ and $a=42$. Further $P(\Delta (m, k))=47$ and $43|\Delta (m, k)$. If
$p|\Delta (m, k)$ with $13\le p\le 41$, then $42\notin \B\{47, p\}$ by Corollary \ref{Ircor} $(iii)$.
Thus we may further suppose that $p|\Delta (m, k)$ with
$p\le 11$ or $p\in \{43, 47\}$. Also $P(m)\le 41$ otherwise each of
$P(m), P((m+1)(m+2)), P((m+3)(m+4))$ is $>41$ which is not possible.
Again we get $P(m+2)\le 41$ since otherwise each of $P(m(m+1)), P(m+2), P((m+3)(m+4))$ is $>41$.
Therefore $P(m(m+2))\le 41$ implying $P(m(m+2))\le 11$. If $m$ is odd, then $m=N-2$ for
$N\in N_2(11)$ and we check that there is a prime $p>11, p\notin \{43, 47\}$ with $p|\Delta (m, k)$
which is a contradiction. Thus $m$ is even and we have $P(\frac{m}{2}(\frac{m}{2}+1))\le 11$ implying
$m=2N-2$ with $N\in N_1(11)$. This is again not possible as above.

Let $k=3$. Then $P(\Delta(m, k))\le 53$ by Corollary \ref{Ircor} $(i)$. Recall that
$P_1>P_2>\cdots \geq k+2$ are all the primes dividing $\Delta(m, k)$. We observe that
$P_1>41$ since $m+j\notin N_1(41)$ for $1\le j<k$. Further $P((m+1)(m+2))>41$ if $P(m)>41$ and
$P(m(m+1))>41$ if $P(m+2)>41$ which are excluded by Corollary \ref{Ircor} $(iii)$ as
above. Thus we may suppose that $P_1=P(m+1)>41$ and $P(m(m+2))\le 41$. If $m$ is even,
then $m=2N-2$ for $N\in N_1(41)$ and we check that either $P_1>53$ or $a>50$ for
$a\in \B\{P_1, P_2, \ldots \}$. Thus $m$ is odd. If $P(m(m+2))\le 31$, then
$m=N-2$ with $N\in N_2(31)$ and we check that either $P_1>53$ or $a>50$ for
$a\in \B\{P_1, P_2, \ldots \}$ which is excluded. Thus $P_2=P(m(m+2))\in \{37, 41\}$ which
together with $41<P_1\le 53$ implies $a>50$ for $a\in \B\{P_1, P_2\}$ except when
$P_1=43, P_2=41$ where $a=40\in \B\{P_1, P_2\}$. Thus $a=40, P(m+1)=43$ and
$P(m(m+2))=41$. Further by Corollary \ref{Ircor} $(iii)$, we may assume
$p\in \{2, 3, 7, 41, 43\}$ for $p|\Delta(m, 3)$ and $2\cdot 43|(m+1)$.  By looking
at the possible prime factorisations of $m, m+1, m+2$ and taking $(m+2)-m$ or
$m-(m+2)$, we have the following possibilities.
\begin{align*}
\begin{array}{cll}
&m+1=2^r\cdot 7^y\cdot 43^t, \ & 3^x-41^z=\pm 2;\\
&m+1=2^r\cdot 3^x\cdot 43^t, \ & 7^y-41^z=\pm 2;\\
&m+1=2^r\cdot 43^t, \ \ & 3^x-41^z=\pm 2;\\
&m+1=2^r\cdot 43^t, \ \ & 3^x\cdot 7^y-41^z=\pm 2;\\
&m+1=2^r\cdot 43^t, \ \ & 3^x-7^y\cdot 41^z=\pm 2;\\
&m+1=2^r\cdot 43^t, \ \ & 7^y-3^x\cdot 41^z=\pm 2;\\
\end{array}
\end{align*}
where $r, x, y, z, t$ are positive integers. The second and fourth equations are
excluded by taking remainders modulo $7$. Calculating  modulo $8$ for the
remaining possibilities, we get the following four simultaneuos equations.
\begin{align*}
\begin{array}{|clll|}\hline
C1: &3^x-41^z=2, \ & 3^x-2^r\cdot 7^y\cdot 43^t=1, \ & 2^r\cdot 7^y\cdot 43^t-41^z=1, \ x \ {\rm odd}\\ \hline
C2: &3^x-41^z=2, \ & 3^x-2^r\cdot 43^t=1, \ & 2^r\cdot 43^t-41^z=1, \ x \ {\rm odd}\\ \hline
C3: &3^x-7^y\cdot 41^z=2, \ & 3^x-2^r\cdot 43^t=1, \ & 2^r\cdot 43^t-7^y\cdot 41^z=1\\ \hline
C4: &3^x\cdot 41^z-7^y=2,  \ &  3^x\cdot 41^z-2^r\cdot 43^t=1, & 2^r\cdot 43^t-7^y=1\\ \hline
\end{array}
\end{align*}
If $4|2^r$ in $C2$, we get a contradiction by taking remainders modulo $4$ since $x$ is odd,
thus $2^r=2$. Calculating modulo $7$ in all the possibilities, we find that $C1$ is excluded
since $x$ is odd. Further $6|(x-1)$ in $C2$; $6|(x-2)$, $3|r$ in $C3$ and $3|r$ in $C4$.
Note that $x\ge 2$. Taking remainders modulo $9$ again, we find that $3|(z+1)$ in $C2$; $3|t$ in
$C3$ and $3|t, 3|(y-1)$ in $C4$. Thus we have
$(-41^{\frac{z+1}{3}})^3+3\cdot 41(3^{\frac{x-1}{3}})^3=2\cdot 41$
in $C2$, $(-2^{\frac{r}{3}}\cdot 43^{\frac{t}{3}})^3+9(3^{\frac{x-2}{3}})^3=1$ in $C3$ and
$(2^{\frac{r}{3}}\cdot 43^{\frac{t}{3}})^3+7(-7^{\frac{y-1}{3}})^3=1$ in $C4$. We solve the
Thue equations $X^3+123Y^3=82, X^3+9Y^3=1$ and $X^3+7Y^3=1$ with $X, Y$ integers in {\bf PariGp}
to find that it is not possible.

We recall that Theorem \ref{a<50} follows from Theorem \ref{a3k} when $k\ge 10$. Therefore we prove Theorem \ref{a3k}
with $k\ge 10$ in Sections 7, 8 and this will complete the proofs of Theorems \ref{a3k} and \ref{a<50}.

\section{Proof of Theorem \ref{a3k} for $k\ge 10$}\label{k<5000I}

We may suppose by Corollary \ref{Ircor} $(i)$ that $P(\Delta(m, k))\le a+k\le 6k$.
Let $k\le 17$. We may suppose that $\max(30, 1.5k)<a\le 5k$. First assume that $m+j\notin N_1(41)$
for any $1\le j<k$. Let
\begin{align*}
\mathfrak{L}_i(k, a):=\{p: \max(41, \frac{a}{i})<p\le \frac{a+k}{i}\} \ \ {\rm for} \ \ 1\le i\le 5
\end{align*}
and $\ell(k):=\underset{1.5k<a\le 5k}{\max}|\cup^5_{i=1}\mathfrak{L}_i(k, a)|$. There are at most
$\ell(k)$ primes $>41$ dividing $\Delta(a+1, k)$ and we delete numbers in $\{m, m+1, \cdots, m+k-1\}$
divisible by those primes. We are left with at least $k-\ell(k)$ numbers. We observe that the prime
factors of each of these numbers are at most $41$ otherwise the assertion follows by
Lemma \ref{irupdate}. We call $U$ the largest such number. From \cite[Tables IA]{lehmer}, we may assume
that each of these numbers is at least at a distance $2$ from the preceding one. Thus $m+k-1\ge U\ge m+2(k-\ell(k)-1)$.
Hence we have a contradiction if $k-2\ell(k)-1>0$. This is the case since $\ell(k)=2, 3, 4, 5$ when
$k=10, k\in \{11, 12\}$, $k\in \{13, 14\}$, $k\in \{15, 16, 17\}$, respectively. Therefore we suppose that
$m+j_0=N\in N_1(41)$ for some $1\le j_0\le k-1$. Then $\Delta(m, k)=M_{j_0}(N, k)$.  We check that $P(M_j(N, 7))>102$
for $1\le j<7$ when $N>10000$ and $N\in N_1(41)$. Thus $m<N\le 10000$. For each $m<10000$, we check that
$P(\Delta(m, 10))>102$ for $m\ge 118$. Therefore $P(\Delta(m, k))>6k$ when $m\ge 118$. Further we
find that $p_{i+1}-p_i\le 10$ for $p_i<118$. Hence for $m<118$, $P(\Delta(m, k))\ge m$ since $k\ge 10$.
Therefore we have $P(\Delta(m, k))\ge \min(m, 6k+1)>k+a$ for all $m$. Now the assertion follows by Corollary \ref{Ircor} $(i)$.

Thus $k\ge 18$. First we check that $\om_1<k$ for $k\le 100$ which together with
\eqref{om2} and Corollary \ref{<4k8000} implies $\om_1<k$ for all $k$. Suppose $m\le 10^{10}$.
If at least one of $m, m+1, \ldots, m+k-1$ is a prime, then $P(\Delta(m, k))\ge m>k+a$ and
therefore the assertion follows from Corollary \ref{Ircor} $(i)$. Hence we may suppose that each of
$m, m+1, \ldots, m+k-1$ is composite. By Lemma \ref{grim}, we obtain $\om(\Delta(m, k))\ge k>\om_1$
which contradicts \eqref{1+om1}. Therefore we have $m>10^{10}$ which implies $k>500$ by
\eqref{1+om1} and \eqref{Lbd} with $t_0=\om_1$.

By \eqref{1+om1} and \eqref{om2}, we have $\om(\Delta (m, k))\le \om_2$.
We obtain from \eqref{E0}, Lemma \ref{pix} $(vi)$ and $k>500$ that
\begin{align}\label{E1}
m^{k-\om_2}<(k-1)!=\frac{k!}{k}\le \frac{\sqrt{2\pi k}}{k} \left(\frac{k}{e}\right)^k
e^{\frac{1}{12k}}<\left(\frac{k}{e}\right)^k.
\end{align}
Since $m\ge 10^{10}$, we get
\begin{align*}
\log k -1>\frac{(k-\om_2)\log m}{k}\ge 10(\log 10)(1-\frac{\om_2}{k}).
\end{align*}
By using estimates of $\pi(\nu)$ from Lemma \ref{pix} $(i)$, we obtain
\begin{align*}
k>e^{\left(1+10(\log 10)\left(1-\frac{\frac{6}{5}}{\log \frac{6k}{5}}
\left(1+\frac{1.2762}{\log \frac{6k}{5}}\right)-
\sum^4_{j=1}\frac{1}{j\log \frac{k}{j}}\left(1+\frac{1.2762}{\log \frac{k}{j}}\right)\right)\right)}=:J(k)
\end{align*}
Since $J(k)$ is an increasing function of $k$ and $k>500$, we have $k>J(500)\ge 4581$. Further
$k>J(4581)\ge 578802$ and hence $k>J(578802)>4.5\times 10^7$.
Let $m\le 123k$. Then, by Corollary \ref{n<<}, there is a
prime $P_1\ge m$ such that $P_1|\Delta(m, k)$. Since $m>a+k$, the assertion follows
by Corollary \ref{Ircor} $(i)$. Therefore we may suppose that $m>123k$.

Assume that $m+k-1\ge k^{\frac{3}{2}}$. Then $m>\frac{k^{\frac{3}{2}}}{e}$ and we get
from \eqref{E1} and Corollary \ref{<4k8000} that
\begin{align*}
k^k>(k^{\frac{3}{2}})^{k-\pi(4k)}
\end{align*}
which together with estimates of $\pi(\nu)$ from Lemma \ref{pix} implies
\begin{align*}
0>\frac{k-3\pi(4k)}{k}\ge 1-\frac{12}{\log 4k}
\left(1+\frac{1.2762}{\log 4k}\right).
\end{align*}
The right hand expression is an increasing function of $k$ and the inequality does not
hold at $k=10^6$. Therefore $m+k-1<k^{\frac{3}{2}}$. By Lemma \ref{k3/2}, we get
\begin{align*}
\binom{m+k-1}{k}\leq (2.83)^{k+k^{\frac{3}{4}}}k^{\frac{3}{2}(\pi(4k)-\pi(k))}
\end{align*}
since $|\{i: P(m+i)\le k\}|\ge k-(\pi(4k)-\pi(k))$ by \eqref{piDk} and Corollary \ref{<4k8000}.
On the other hand, we have $m>123k$ implying
\begin{align*}
\binom{m+k-1}{k}\geq \binom{124k}{k}=\frac{(124k)!}{k!(123k)!}
&>\frac{\sqrt{2\pi (124k)}(\frac{124k}{e})^{124k}}
{\sqrt{2\pi k}(\frac{k}{e})^k{\rm e}^{\frac{1}{12k}}
\sqrt{2\pi (123k)}(\frac{123k}{e})^{123k}{\rm e}^{\frac{1}{12\cdot 123k}}}\\
&>\frac{0.4}{\sqrt{k}}e^{-\frac{1}{8k}}(335.7)^k
\end{align*}
using estimates of $\nu!$ from Lemma \ref{pix}. Comparing the upper and lower bounds, we obtain
\begin{align*}
0>\log (0.4)-\frac{1}{8k}-0.5\log k+k\log (\frac{335.7}{2.83})-k^{\frac{3}{4}}\log (2.83)
-\frac{3}{2}(\pi(4k)-\pi(k))\log k.
\end{align*}
By using estimates of $\pi(\nu)$ from Lemma \ref{pix} again, we obtain
\begin{align*}
\frac{(\pi(4k)-\pi(k))\log k}{k}&\le \frac{4\log k}{\log 4k}(1+\frac{1.2762}{\log 4k})-\frac{\log k}{\log k-1}\\
&\le 4\left(1-\frac{\log 4}{\log 4k}\right)\left(1+\frac{1.2762}{\log 4k}\right)-1\\
&\le 4\left(1-\frac{\log 4 -1.2762}{\log 4k}\right)-1\le 3.
\end{align*}
Therefore we have
\begin{align*}
0>\frac{\log (0.4)-\frac{1}{8k}-0.5\log k}{k}+\log (\frac{335.7}{2.83})-k^{-\frac{1}{4}}\log (2.83)-4.5.
\end{align*}
The right hand side of the above inequality is an increasing function of $k$ and the inequality is
not valid at $k=10^6$. This is a contradiction.
\qed

\section{Proof of Theorem \ref{a12}}\label{Proofa12}

By Theorem \ref{a<50}, we restrict to those triplets $(n, a, k)$ given in
the statement of Theorem \ref{a<50} with $a\le 12$. We now factorize $f_{n, a}(x)$ with
$a_0a_n=\pm 1, a_1=a_2=\ldots =a_{n-1}=1$ to find that these $f_{n, a}(x)$ are irreducible.
Hence the assertion follows.
\qed

\section{Proof of Theorem \ref{Laga40}}\label{ProofLa}

For the proof of Theorem \ref{Laga40}, we put $\alpha =a$ throughout this section. As remarked in
Section 1 after the statement of Theorem \ref{Laga40}, we may assume
that $10<a\le 40$. For $n\le 18$ and $n\in \{24, 25, 27, 30, 32, 36, 45, 48, 54, 60, 64, 72, 75, 80, 90, 112, 120\}$,
we find that $L^{(a)}_n(x)$ is irreducible except for $(n, a)$ listed in Theorem \ref{Laga40}.
Thus we assume $n>18$, $n\notin \{24, 25, 27, 30, 32, 36, 45, 48, 54, 60, \\ 64, 72, 75, 80, 90, 112, 120\}$.
Assume that $L^{(\al)}_n(x)$ is reducible. Then $L^{(\al)}_n(x)$ has a factor of degree $k$ with
$1\le k\le \frac{n}{2}$. First we prove the following lemma.

\begin{lemma}\label{2klag}
Let $k\ge 2$. Then $L^{(a)}_n(x)$ has no factor of degree $k$.
\end{lemma}

\begin{proof}
Let $k\ge 2$ and $a\le 40$ if $k=2$. We may restrict to those $(n, k, a)$ given in the list of
exceptions in Theorem \ref{a<50}. For each of these triplets $(n, k, a)$, we first check
if there is a prime $p\ge k+2$ with \eqref{Lagcond} such that either \eqref{2u0} or \eqref{2k} is
satisfied and they can be excluded by Lemma \ref{irupdate}. We are now left with triples $(n, k, a)$
given by $k=2, (n, a)\in \{(100, 21), (40, 24), (256, 33), (42, 40)\}$. For these $(n, a)$, we
check that $L^{(a)}_n(x)$ is irreducible.

Let $k=2$ and $40<a\le 50$. Suppose $n\notin N_1(23)$ and $n+a\notin N_1(23)$. Then $P_1=P(n(n-1))>23$
and $P_2=P((n+a)(n+a-1))>23$. Further either $P_1\nmid (a+1)(a+2)$ or $P_2\nmid (a+1)(a+2)$ and
then the assertion follows by Lemma \ref{irupdate}. Therefore we may assume that either $n=N\in N_1(23)$
or $n+a=N\in N_1(23)$. Further we may also  suppose that $P(n(n-1)(n+a)(n+a-1))\le P((a+1)(a+2))$
since otherwise the assertion follows by Lemma \ref{irupdate}.
For $N\in N_1(23)$ and $N>10000$, we check that $P((N-a)(N-a-1))>P((a+1)(a+2))$ and
$P((N+a)(N+a-1))>P((a+1)(a+2))$ except when $(a, N)\in \{(45, 10648), (46, 12168)\}$ where
$P(N(N-1))\in \{13, 23\}$, respectively. Observe that $N(N-1)|n(n-1)(n+a)(n+a-1)$. By taking
$p=P(N(N-1))$, the assertion follows from Lemma \ref{irupdate}. We now consider $n\le 10000$.
Let $a$ be given. By Lemma \ref{irupdate}, we first restrict to those $n$ for which
$P(n(n-1)(n+a)(n+a-1))\le P((a+1)(a+2))$. Further we check that there is a prime
$p|n(n-1)(n+a)(n+a-1), p>7$ and $p\nmid (a+1)(a+2)$. Lemma \ref{irupdate} implies the assertion now.
\end{proof}

By Lemma \ref{2klag}, we only need to consider $k=1$. If there is a prime $p|n(n+a), p\nmid (a+1)$ with either
$p\ge 11$ or $p=7, a\neq 47$ or $p=5, a\notin \{23, 48\}$ or $p=3, a\notin \{16, 24, 25, 34, 43\}=:S_1$, then
the assertion follows by Lemma \ref{irupdate} and Corollary \ref{ircor3}. Let $P_a=\{2\}\cup \{p: p|(a+1)\}$ if
$a\notin S_1\cup \{23, 47, 48\}$, $P_a=\{2, 3\}\cup \{p: p|(a+1)\}$ if $a\in S_1$,  $P_a=\{2, 3, 5\}$ if $a=23$,
$P_a=\{2, 3, 7\}$ if $a=47$ and $P_a=\{2, 5, 7\}$ if $a=48$. Thus for a given $a$, we may assume that $p|n(n+a)$
implies $p\in P_a$.

Let $a$ be given. Let $p|n$ with $p>2$. Then $p\in P_a$. As in the proof of
Lemma \ref{irupdate}, if we have $\phi'_j<1$ for all $1\le j\le n$,
then $L^{(\al)}_n(x)$ does not have a linear factor and we are done.
Let $1\le j\le 50$. We compute $\phi_j$ to find that $\phi_j<1$ for $j>1$ except when
$(p, a)\in T_1:=\{(3, 16), (3, 17), (3, 34), (3, 35), (3, 43), (3, 44), (5, 23), (5, 24),
(5, 48), (5, 49), (7, 47), (7, 48)\}$ where $\phi_j<1$ for $j>2$ and except when
$23\le a\le 26, p=3$ where $\phi_j<1$ for $j>4$. Let $j>50$. By using ord$_p(s!)\le \frac{s}{p-1}$,
we find that
\begin{align*}
\phi_j=\frac{{\rm ord}_p((a+j)!)-{\rm ord}_p(a!)}{j}
\le \frac{\frac{a+j}{p-1}-{\rm ord}_p(a!)}{j}\le \frac{1}{p-1}+\frac{\frac{a}{p-1}-{\rm ord}_p(a!)}{51}<1.
\end{align*}
It suffices to show that $\phi'_1<1$ except when $(p, a)\in T_1$ for which we need
to show $\phi'_j<1$, $1\le j\le 2$ and except when $23\le a\le 26, p=3$ for which
we need to show $\phi'_j<1$ for $1\le j\le 4$. Let $\phi'_0=\max\{\phi'_i\}$ for $1\le i\le 4$.
It suffices to show $\phi'_0<1$ is always valid. This is the case except when
$a\in \{24, 49\}, p=5$; $a\in \{17, 24, 25, 26, 35, 44\}, p=3$ and $a=48, p=7$.
Further ord$_5(n)\le 1$ when $a\in \{24, 49\}$,  ord$_7(n)\le 1$ when $a=48$, ord$_3(n)\le 1$
when $a\in \{17, 24, 25, 35, 44\}$ and ord$_3(n)\le 2$ when $a=26$
otherwise $\phi'_0<1$. Let $a\in \{17, 26, 35\}$ and ord$_3(n)=1$ or ord$_3(n)=2$.
Then from $n(n+a)=2^{\al}3^{\be_3}$ and \emph{gcd}$(n, n+a)\le 2$, we obtain
$n\in \{3, 6, 9, 18\}$ which is not possible. Let $a=49$ and ord$_5(n)=1$. Then
from $n(n+a)=2^{\al}5^{\be_5}$ and gcd$(n, n+a)=1$, we obtain $n=5$ which
is again not possible. Here gcd$(a, b)$ stands for greatest common divisor of
$a$ and $b$.

Therefore $n$ is a power of $2$ except when $a=24$ where ord$_3(n)\le 1$ or
ord$_5(n)\le 1$; $a=25$ where ord$_3(n)\le 1$; $a=44$ where ord$_3(n)\le 1$ and
$a=48$ where ord$_7(n)\le 1$. From the definition of $P_a$, we observe that
$n(n+a)$ has at most two odd prime factors except when $a=34$ where it has
at most three odd prime factors. Hence we always have $n, n+a$ of the form
\begin{align}\label{nn+a}\begin{array}{lll}
n=2^{\al+\del}, \ \frac{n+a}{2^{\del}}=&p^{\be_p} \ &{\rm if} \ P_a=\{2, p\}\\
n=2^{\al+\del}, \ \frac{n+a}{2^{\del}}\in &\{p^{\be_{p_1}}_1, p^{\be_{p_2}}_2,
p^{\be_{p_1}}_1p^{\be_{p_2}}_2\} \ &{\rm if} \ P_a=\{2, p_1, p_2\}\\
n=2^{\al+\del}, \ \frac{n+a}{2^{\del}}\in &\{p^{\be_{p_1}}_1, p^{\be_{p_2}}_2,
p^{\be_{p_3}}_3, p^{\be_{p_1}}_1p^{\be_{p_2}}_2,   p^{\be_{p_1}}_1p^{\be_{p_3}}_3, & \\
&p^{\be_{p_2}}_2p^{\be_{p_3}}_3, p^{\be_{p_1}}_1p^{\be_{p_2}}_2p^{\be_{p_3}}_3\} \ &{\rm if} \
P_a=\{2, p_1, p_2, p_3\}.
\end{array}\end{align}
where $2^{\del}||a$ and in addition $n, n+a$ is of the form
\begin{align}\label{na35}\begin{array}{llll}
&n=15\cdot 2^{\al+3}, \ &n+a=8\cdot 3^{\be_3+1}  \ \ \ {\rm or} &\\
&n=3\cdot 2^{\al+3}, \ &n+a\in \{8\cdot 3^{\be_3+1}, 8\cdot 3^{\be_3+1}5^{\be_5}\}
\ \ &{\rm if} \ a=24\\
&n=3\cdot 2^{\al}, \ &n+a=13^{\be_{13}} \ \ &{\rm if} \ a=25\\
&n=3\cdot 2^{\al+2}, \ &n+a=4\cdot 5^{\be_{5}} \ \ &{\rm if} \ a=44\\
&n=7\cdot 2^{\al+4}, \ &n+a=16\cdot 5^{\be_{5}} \ \ &{\rm if} \ a=48.
\end{array}
\end{align}
Here all the exponents of odd prime powers appearing in \eqref{nn+a} and \eqref{na35} are positive. For $n<512$ and $n$ of the
form given by \eqref{nn+a} or \eqref{na35} which are given by $n\in \{96, 128, 192, 224, 240, 256, 384, 448, 480\}$, we
check that there is a prime $p|(n+a), p\notin P_a$ except when $(n, a)\in \{(256, 14), (128, 16), (256, 16), \\
(96, 24), (192, 24), (256, 32), (256, 33), (128, 34)\}$.
We find that for each of these $(n, a)$, the polynomial $L^{(a)}_n(x)$ is irreducible. Therefore we have $n\ge 512$.

From the equality $\frac{n+a}{2^{\del}}-\frac{n}{2^\del}=\frac{a}{2^{\del}}$, we obtain an equation of the form
\begin{align*}
p^{\be_p}-2^\al=\frac{a}{2^{\del}} \ \ {\rm or} \ \ p^{\be_{p_1}}_1p^{\be_{p_2}}_2-2^\al=\frac{a}{2^{\del}}
\end{align*}
or further $3^{\be_{3}}5^{\be_{5}}7^{\be_{7}}-2^\al=17$ (only when $a=34$) or $3^{\be_3}-5\cdot 2^\al=1$ (only when $a=24$)
or $13^{\be_{13}}-3\cdot 2^\al=25$ (only when $a=25$) or $5^{\be_{5}}-3\cdot 2^\al=11$ (only when $a=44$) or
$5^{\be_{5}}-7\cdot 2^\al=3$ (only when $a=48$).
In each of the equations thus obtained, we note that $8|2^\al$ since $n\ge 512$.
Out of all the equations, we need to consider only those which are valid under remainders modulo $8$
and hence we restrict to those. Here we use $p^{\be_p}\equiv 1$ or $p$ modulo $8$ according as $\be_p$ is even or odd, respectively.
They are now expressed as the Thue equation
\begin{align*}
X^3+AY^3 = B
\end{align*}
and we solve them in {\bf PariGp}.
For instance, let $a=32$. Then we obtain equations of the form $3^{\be_3} -2^\al=1$,
$11^{\be_{11}}-2^\al=1$, $3^{\be_3} 11^{\be_{11}}-2^\al=1$. By taking remainders modulo $8$, we find that
$\be_3, \be_{11}, \be_3+\be_{11}$ are even for the first, second and third equation, respectively.
This implies $3^{\frac{\be_3}{2}}-1=2, 3^{\frac{\be_3}{2}}+1=2^{\al-1}$ giving $3^{\be_3}=9, 2^{\al}=8$
for the first equation and $11^{\frac{\be_{11}}{2}}-1=2, 11^{\frac{\be_{11}}{2}}+1=2^{\al-1}$ giving
a contradiction for the second equation. Observe that $2^{\al}>8$ since $n\ge 512$. Thus we are left
with $3^{\be_3} 11^{\be_{11}}-2^\al=1$. For some $0\le r, s, t\le 2$, we have $\al+r, \be_3-s, \be_{11}-t$
all are multiples of $3$ and from $-2^{\al +r}+2^r3^s11^t3^{\be_3-s} 11^{\be_{11}-t}=2^r$, we obtain the
Thue equations $X^3+AY^3=B$ with $B=2^r, A=2^r3^s11^t, 0\le r, s, t\le 2$ and with $X$ a power of
$2$ and $33|AY$. There are $27$ possibilities of pairs $(A, B)$. If $A=1$, then $B=1$ and we factorise
$X^3+Y^3$ to get a contradiction. Thus the case $A=1$ is excluded. For all other values of $(A, B)$
than those given by $t=2$, we check in {\bf PariGp} that none of the solutions $(X, Y)$ of
Thue equations thus obtained
satisfy the condition $X$ a power of $2$ and $33|AY$ except when $A=66, B=2$ where $X=-4$ and $Y=1$ from
which we obtain $n=1024$. When $t=2$, from
$3^{\be_3-s+3}11^{\be_{11-2+3}}-2^{3-r}3^{3-s}\cdot 11 \cdot 2^{\al+r-3}=3^{3-s}\cdot 11$, we obtain
the Thue equations $X^3+AY^3=B$ with $B=3^{3-s}\cdot 11, A=2^{3-r}3^{3-s}\cdot 11, 0\le r, s\le 2$ and
$33|X$ and $Y$ a power of $2$. We check again in {\bf PariGp} that none of the solutions $(X, Y)$ of these
Thue equations thus satisfy the condition $33|X$ and $Y$ a power of $2$. Hence we need to consider
$n=1024$ when $a=32$. For another example, let $a=48$. We obtain
equations of the form $5^{\be_5} -2^\al=3$, $7^{\be_{7}}-2^\al=3$, $5^{\be_5} -7\cdot 2^\al=3$ and
$5^{\be_5}7^{\be_{7}}-2^\al=3$. The first three equations are excluded modulo $8$ and for the last equation, we
find that $\be_5, \be_7$ are both odd. Taking remainders modulo $7$ imply $3|(\al-2)$ or $3|(\al+1)$
and hence from the equation $-2^{\al+1}+2\cdot 5^{\be_5}7^{\be_{7}}=6$, we obtain the Thue equations
$X^3+AY^3=B$ with $B=6, A=2\cdot 5^s7^t, 0\le s, t\le 2$ and $X$ a power of $2$ and $70|AY$.
When $t=2$, from $5^{\be_5-s+3}7^{\be_7+1}-4\cdot 5^{3-s}\cdot 7\cdot 2^{\al-2}=3\cdot 5^{3-s}\cdot 7$, we
obtain the Thue equations $X^3+AY^3=B$ with $B=21\cdot 5^{3-s}, A=28\cdot 5^{3-s}, 0\le s\le 2$ and
$35|X$ and $Y$ a power of $2$. We check in {\bf PariGp} that all the solutions $(X, Y)$ of these Thue
equations are excluded except when $(A, B)=(70, 6)$ where $X=-4, Y=-1$ and we obtain $n=512$. Hence
we need to consider $n=512$ when $a=48$. Similarly, all other $a$'s are excluded except when
$a\in \{20, 24\}$ where we obtain $(n, a)\in \{(4096, 20), (1920, 24)\}$.

Thus we now exclude the cases $(n, a)\in \{(4096, 20), (1920, 24), (1024, 32), (512, 48)\}$. We take $p=2$
and show that $\phi'_j<1$ for all $1\le j\le n$. This is shown by checking ord$_2(\Delta_j)-$ord$_2(\binom{n}{j})<j$
for $j$ such that ord$_2(\Delta_j)\ge j$ for these pairs $(n, a)$. Hence they are all excluded.
\qed

\section*{Acknowledgments}

A part of this work was done when the second author was visiting Max-Planck Institute for Mathematics
in Bonn during Feburary-April, 2009 and he would like to thank the MPIM for the invitation and the
hospitality. We also thank the referee for his comments on an earlier draft.

\end{document}